\RequirePackage{fix-cm}     
\documentclass[natbib]{svjour3}                  
\smartqed  
\usepackage{mathptmx}      

\usepackage{amsmath,amssymb}
\usepackage{enumerate}

\newcommand{\R}{\mathbb{R}}
\newcommand{\es}{\varnothing}
\newcommand{\st}{{\;\vrule height10pt width0.7pt depth3pt\;}}
\newcommand{\CCC}{\mathcal{C}}
\newcommand{\DDD}{\mathcal{D}}
\newcommand{\NNN}{\mathcal{N}}

\journalname{DCG}
\begin{document}
  
\begin{center}
\noindent\textbf{Any Finite Group is the Group of Some Binary, Convex Polytope}

\bigskip

\textsc{Jean-Paul Doignon}\\[2mm]
{Universit\'e Libre de Bruxelles,\\
D\'epartement de Math\'ematique, c.p.~216,\\
1050~Bruxelles, Belgium.\\
\texttt{doignon@ulb.ac.be}}
\end{center}

\medskip

\noindent\textbf{Abstract} 
For any given finite group, Schulte and Williams (2015) establish the existence of a convex polytope whose combinatorial automorphisms form a group isomorphic to the given group.  We provide here a shorter proof for a stronger result: the convex polytope we build for the given finite group is binary, and even combinatorial in the sense of Naddef and Pulleyblank (1981); the diameter of its skeleton is at most 2; any combinatorial automorphism of the polytope is induced by some isometry of the space; any automorphism of the skeleton is a combinatorial automorphism.

\medskip

\noindent\textbf{Keywords} Automorphism, group, graph, convex polytope 

\medskip
 
\noindent\textbf{AMS-classification} 05E18; 52B15 

\thispagestyle{plain}
\markboth{J.-P. DOIGNON}{A Binary Convex Polytope for any Finite Group}

\section{Introduction}  \label{sect_Introduction}
A \textsl{combinatorial automorphism} of a convex polytope is a permutation of its set of vertices that maps the set of vertices of any face to the set of vertices of some face.
Recently, \citet{Schulte_Williams_2015} established that any finite group is isomorphic to the group of combinatorial automorphisms of some convex polytope.  Their proof is rather long and involved.  We propose here a shorther one which even establishes a stronger result because the convex polytope we build for a given group is always \textsl{binary}: all its vertices have coordinates equal to $0$ or $1$.  The polytope is moreover \textsl{combinatorial} in the sense of \cite{Naddef_Pulleyblank_1981}: it is a binary polytope on which every two nonadjacent vertices have their midpoint equal to the midpoint of two other vertices (such combinatorial polytopes appear also, for instance, in \citealp{Matsui_Tamura_1995}).

In all the paper, we consider $\R^d$ as a euclidean vector space with the usual dot product.  The \textsl{isometries} of $\R^d$ are the permutations of $\R^d$ that preserve distance.  For a subset $P$ of $\R^d$, the \textsl{isometries of} $P$ are the isometries of $\R^d$ that stabilize $P$.

\begin{theorem}\label{theorem}\sl
For each finite group $G$ there exist a natural number $d$ and a convex polytope $P_G$ in $\R^d$ satisfying Properties~\textrm{(i)--(v)}:
\begin{enumerate}[\quad\rm(i)]\itemsep0.3em
\item the group of combinatorial automorphisms of the convex polytope $P_G$ is isomorphic to $G$;
\item any combinatorial automorphism of $P_G$ is the restriction to the set of vertices of some isometry of $P_G$;
\item the convex polytope $P_G$ is binary, and even combinatorial;
\item the diameter of the skeleton (graph) of $P_G$ is at most $2$;
\item any automorphism of the skeleton of $P_G$ is a combinatorial automorphism of $P_G$.
\end{enumerate}
\end{theorem}

We denote by $G$ a given, finite group.  
The required polytope $P_G$ for $G$ arises from a two-step construction.  
In Section~\ref{sect_Graph_from_Group} we recall results on the existence of a graph $\Gamma(G)$  whose automorphism group is isomorphic to $G$.  
In Section~\ref{sect_Polytope_from_Graph} we build for almost any graph $\Gamma$ 
a polytope $P(\Gamma)$ whose combinatorial automorphism group is isomorphic to the automorphism group of the graph $\Gamma$.  It happens that $P(\Gamma)$ is a combinatorial polytope; moreover, any combinatorial automorphism of the polytope $P(\Gamma)$ is the restriction of some isometry of $P(\Gamma)$ (Proposition~\ref{propF}). The proof also shows that all isometries of $P(\Gamma)$ are in fact  coordinate permutations of the space $\R^d$ in which we build $P(\Gamma)$.
Incidentally, we notice that any graph on $d$ nodes is an induced subgraph of the graph of some binary, combinatorial, convex polytope of dimension $d$.

For any given group $G$, the two-step construction produces a convex polytope $P_G=P(\Gamma(G))$ which satisfies Properties~(i)--(iv) in Theorem~\ref{theorem} (as we will see in Section~\ref{sect_Polytope_from_Group}, we might need to replace $\Gamma(G)$ with its complement graph; also we handle the two smallest groups $G$ in another way).
In Section~\ref{sect_Polytope_from_Group}, after completing the proof, we discuss how to decrease the dimension and/or the number of vertices of such a polytope $P_G$.  Next, we indicate how to construct the polytope $P_G$ in order that it moreover satisfies Property~(v).  We also mention a curious, immediate consequence: 
for any finite group $G$, there is a directed graph for which the asymmetric travelling-salesman polytope has its automorphism group  isomorphic to $G$. 

Let us mention in passing that \cite{Babai_1977} characterizes the isometry groups of convex polytopes (binary or not) which are transitive on the set of vertices.  The only excluded groups are the generalized dicyclic groups and the abelian groups of exponent at least 2 (see \citealp{Babai_1977}, or \citealp{Babai_Godsil_1982}, for a definition of generalized dicyclic groups and, for instance, \citealp{Robinson_1996} for further group terminology). 

\section{Building a Graph for any Group}
\label{sect_Graph_from_Group}

Our graphs $\Gamma=(V,E)$ have neither loops nor multiple links.  We use the terms `node' and `link' for an element of respectively $V$ and $E$ (while keeping `vertex' and `edge' for polytopes).  Let again $G$ be a finite group.  \cite{Frucht_1939} was the first to show the existence of some graph whose automorphism group is isomorphic to $G$ (for a historical perspective, see \citealp{Hevia_1995}).  For reasons which will become clear in Section~\ref{sect_Polytope_from_Group}, we would like to select, given the group $G$, a graph having $G$ as its automorphism group and moreover having the least possible number of nodes.  

\cite{Babai_1974} proves that, with the exception of the three cyclic groups $C_3$, $C_4$ and $C_5$ (of orders 3, 4 and 5), there exists, for any finite group $G$, a graph having at most $2\,|G|$ nodes with automorphism group $G$.  For many finite groups $G$, there is a graph on $|G|$ nodes or less with the stronger property that its automorphism group is isomorphic to $G$ and acts regularly on the set of nodes (we will not make use of the latter property).  As found by \cite{Godsil_1978} and reported by \cite{Babai_Godsil_1982}, such a graph exists for any finite group $G$ except for 
\begin{enumerate}[\qquad-~]\itemsep0.3em 
\item the abelian groups of exponent at least 3;
\item the generalized dicyclic groups;
\item 13 other groups of orders at most 32.
\end{enumerate}
Now forgetting about regularity, we mention that \cite{Arlinghaus_1985} provides for each finite, abelian group $G$ the minimum number of nodes in a graph having $G$ as its automorphism group.

\section{Building a Binary Polytope for any Graph}
\label{sect_Polytope_from_Graph}

For all the section, $\Gamma=(V,E)$ is a graph with at least one node.  For two nodes $u$, $v$ in $V$, we write $u \sim v$ when $u$ and $v$ are linked in $\Gamma$.  For a given node $u$ of $\Gamma$, we denote by $N(u)$ the set of neighbours of $u$ (not including $u$), as well as the graph induced by $\Gamma$ on $N(u)$.
 
We canonically identify $\R^V$ with the euclidean vector space $\R^{|V|}$.  The \textsl{characteristic vector} $\chi(S)$ of a subset $S$ of $V$ is the vector in $\R^V$ defined by
$$
\chi(S)_v \quad=\quad
\begin{cases}
1 & \text{ if } v \in S,\\
0 & \text{ if } v \notin S.
\end{cases}
$$
The \textsl{polytope $P(\Gamma)$ of the graph} $\Gamma=(V,E)$ is the convex hull in $\R^V$ of the characteristic vectors of the empty set, the one-element subsets of $V$ and the links in $E$:
\begin{equation}\label{eqn_P_Gamma}
P(\Gamma) \;=\; \text{conv}\big(\{\chi(\es)\} \cup \{\chi(\{v\}) \st v\in V\} \cup \{\chi(e) \st e\in E\}\big).
\end{equation}
A series of statements now collect properties of the polytope $P(\Gamma)$, starting with a description of its skeleton.

\begin{proposition}\label{prop1}\sl
The vertices of the polytope $P(\Gamma)$ are all the characteristic vectors appearing in Equation~\eqref{eqn_P_Gamma}.  Adjacency among the vertices is as follows:
\begin{enumerate}[\quad\rm(i)]\itemsep0.3em
\item the vertex $\chi(\es)$ is adjacent to each vertex $\chi(\{v\})$, for $v$ in $V$, and to no other vertex;
\item the vertices $\chi(\{u\})$ and $\chi(\{v\})$, for distinct nodes $u$, $v$ in $V$, are adjacent if and only if the nodes $u$ and $v$ are not linked in $\Gamma$ (that is, $u \not\sim v$);
\item two vertices $\chi(\{v\})$ and $\chi(\{u,w\})$, for $v \in V$ and $\{u,w\} \in E$, are adjacent if and only if either $v \in \{u,w\}$ or ($v \notin \{u,w\}$ and neither $v \sim u$ nor $v\sim w$);
\item for distinct links $e$, $f$ in $E$, the vertices $\chi(e)$ and $\chi(f)$ are adjacent if and only if 
either $e$ and $f$ have a common node, or $e$ and $f$ are disjoint and not contained in any four-cycle in the graph $G$.  
\end{enumerate}
\end{proposition}

\begin{proof}
The first assertion follows from the fact that characteristic vectors have all their coordinates equal to $0$ or $1$.  We now consider one by one the cases of adjacency. 
\begin{enumerate}[\quad(i)]\itemsep0.3em
\item $\chi(\es)$ is adjacent to each $\chi(\{v\})$ because adjacency of the same vertices hold on the unit cube.  Moreover, $\chi(\es)$ is not adjacent to any vertex $\chi(e)$, with $e\in E$: indeed, the midpoint of $\chi(\es)$ and $\chi(e)$ is the same as the midpoint of two other vertices, namely $\chi(\{u\}$ and $\chi(\{v\})$ if $e=\{u,v\}$;
\item  if $u \sim v$, then the vertices $\chi(\es)$ and $\chi(\{u,v\})$ have the same midpoint as the vertices $\chi(\{u\})$ and $\chi(\{v\})$ do. So $\chi(\{u\})$ and $\chi(\{v\})$ are not adjacent.
Second, if $u \not\sim v$, then the affine inequality on $R^V$
$$
x_u + x_v - 2\, \sum_{i \in V\setminus\{u,v\}} x_i \;\le\; 1
$$
defines a face of $P(\Gamma)$ which is the segment $[\chi(\{u\}),\chi(\{v\})]$ (meaning the inequality is valid for $P(\Gamma)$ and it is satisfied with equality only at the points of $P(\Gamma)$ which belong to the segment).  Thus the vertices $\chi(\es)$ and $\chi(\{u,v\})$ are adjacent;
\item if $v \in \{u,w\}$, adjacency in $P(\Gamma)$ results from adjacency in the unit cube.  If $v \notin \{u,w\}$, notice that if $u \sim v$, then 
$\chi(\{v\})$ and $\chi(\{u,w\})$ have the same midpoint as $\chi(\{w\})$ and $\chi(\{u,v\})$, so $\chi(\{v\})$ and $\chi(\{u,w\})$ are not adjacent.  The argument is similar if $v \sim w$.  Conversely, if $v \notin \{u,w\}$ and $u \not\sim v$, $v \not\sim w$, the affine inequality 
$$
x_v + \frac12 \, x_u + \frac12 \, x_w - 2\, \sum_{i \in V\setminus\{u,v,w\}} x_i  \;\le\; 1
$$
defines a face of $P(\Gamma)$ which is $[\chi(\{v\}),\chi(\{u,w\})]$;
\item this can be proved with arguments similar to those used in (i)--(iii).  For another proof, consider the \textsl{edge polytope} $P_E = \text{conv}\big( \{\chi(e) \st e\in E\}\big)$ of the graph $\Gamma$ (as  for instance in \citealp{Tran_Ziegler_2014}).  Notice that $P_E$ is the face of $P(\Gamma)$ defined by the inequality $\sum_{i \in V} x_i \le 2$.  Thus the two vertices $\chi(e)$ and $\chi(f)$ of $P(\Gamma)$ are adjacent if and only if they are adjacent in $P_E$, and the latter happens exactly if the condition in (iv) is satisfied (see \citealp{Tran_Ziegler_2014}).
\end{enumerate}
\qed\end{proof}

For a vertex $p$ of the polytope $P(\Gamma)$, let $\NNN(p)$ denote both the neighborhood of $p$ in the skeleton, and the graph induced by the skeleton on the neighborhood.

\begin{corollary}\label{cor1}
There hold for the skeleton of the polytope $P(\Gamma)$:
\begin{enumerate}[\qquad\rm(i)]\itemsep0.3em
\item the diameter is at most $2$;
\item the vertices forming the neighborhood $\NNN(\chi(\es))$ of $\chi(\es)$ are all the vertices $\chi(\{u\})$, for $u \in V$;  the mapping 
$$
V \to \NNN(\chi(\es)):\, u \mapsto \chi(\{u\})
$$ 
is an isomorphism from the complement graph $\bar\Gamma$ to the induced graph $\NNN(\chi(\es))$.
\end{enumerate} 
\end{corollary}

\begin{proof}
Both assertions follow at once from Proposition~\ref{prop1}.
\qed\end{proof}

It follows from Corollary~\ref{cor1}(ii) that any graph $\Gamma$ on $d$ nodes is an induced subgraph of the skeleton of some binary, combinatorial, convex polytope of dimension $d$ (take the polytope $P(\bar\Gamma)$ and consider the graph induced by the skeleton on the neighborhood of $\chi(\es)$).

\begin{corollary}\label{prop2}\sl
The convex polytope $P(\Gamma)$ 
\begin{enumerate}[\qquad\rm(i)]\itemsep0.3em
\item has dimension $|V|$;
\item it is binary, and even combinatorial.
\end{enumerate} 
\end{corollary}

\begin{proof}
Because $\chi(\es)$ is the origin and the vectors $\chi(\{u\})$'s, for $u \in V$, form a base of the vector space $\R^V$, Assertion~(i) holds. Assertion~(ii) was established in the proof of Proposition~\ref{prop1}, except for two nonadjacent vertices of the form $\chi(e)$ and $\chi(f)$ with $e$, $f \in E$.  By (iv) in Proposition~\ref{prop1}, the links $e$ and $f$ are disjoint, and moreover form a four-cycle with two other links $e'$ and $f'$; then $\chi(e)$ and $\chi(f)$ have the same midpoint as $\chi(e')$ and $\chi(f')$ do.
\qed\end{proof}

We now prove that the vertex $\chi(\es)$ is particular among all the vertices of $P(\Gamma)$ exactly when the graph $\Gamma$ satisfies the following mild condition (which we will meet again in later statements):
\begin{quote}
{\rm[*]}~the graph has at least one link ($E\neq\es$) and there does not exist any bipartition of $V$ into two stable sets $C$, $D$ with some node $v_0$ in $C$ linked to all nodes in $D$.
\end{quote}

\begin{lemma}\label{lemma}
The following assertions are equivalent:
\begin{enumerate}[\qquad\rm(a)]\itemsep0.3em
\item all combinatorial automorphisms of the polytope $P(\Gamma)$ fix the vertex $\chi(\es)$;
\item the graph $\Gamma$ satisfies Condition~\textnormal{[*]};
\item the graph $\NNN(p)$ induced on the neighbourhood of a vertex $p$ of $P(\Gamma)$ is isomorphic to the complement graph $\bar\Gamma$ exactly if $p=\chi(\es)$;
\item all automorphisms of the skeleton of $P(\Gamma)$ fix the vertex $\chi(\es)$.
\end{enumerate} 
\end{lemma}

\begin{proof}
(a)~$\implies$~(b). Proceeding by contraposition, we assume that the graph $\Gamma$ does not satisfy [*].  Suppose first that $\Gamma$ has no link.  Then the polytope $P(\Gamma)$ is a simplex with at least two vertices, and so (a) is not valid.  Next, assume $\Gamma$ has a bipartition $C$, $D$ and a node $v_0$ as in [*].
Consider the affine permutation $T$ mapping the point $x$ of $\R^V$ to the point $x'$ of $\R^V$ with, for $u$ in $V$,
$$
x'_u \;=\;
\begin{cases}
1  - \sum_{a \in C} x_a & \text{if } u=v_0,\\
x_u & \text{if } u \neq v_0.
\end{cases}
$$
It is easily checked that $T$ permutes the vertices of the polytope $P(\Gamma)$, more precisely: $T$ exchanges $\chi(\es)$ and $\chi(\{v_0\})$ and, for any node $b$ in $D$, it exchanges $\chi(\{b\})$ with $\chi(\{v_0,b\})$; also, $T$ fixes all other vertices of $P(\Gamma)$ (even those coming from potential links between $C$ and $D$ that do not contain $v_0$).  Thus $T$ induces on the set of vertices a combinatorial automorphism which maps the vertex $\chi(\es)$ on some other vertex, so (a) is not valid.

\smallskip

(b)~$\implies$~(c). By Corollary~\ref{cor1}(ii), $\NNN(\chi(\es))$ is isomorphic to $\bar\Gamma$.  Proceeding again by contradiction, assume some vertex $p$ of the polytope $P(\Gamma)$, different from $\chi(\es)$, is such that the graph induced by the skeleton on its neighbourhood is isomorphic to the graph $\bar\Gamma$.  In particular, the number of neighbours of $p$ must be $|V|$.

If $p=\chi(\{v\})$, for some $v$ in $V$, we know from Proposition~\ref{prop1} that $\chi(\{v\})$ is for sure adjacent to $\chi(\es)$, to $\chi(\{u\})$ when $v \not\sim u$ and to $\chi(\{v,w\})$ when $v \sim w$ (a total number of $|V|$ vertices).  Hence $\chi(\{v\})$ cannot be adjacent to any further vertex.  In particular, there is no link between any two nodes of $\Gamma$ not linked to $v$.  Then the neighborhood of $\chi(\{v\})$ is the union of two cliques, respectively the clique $\CCC$ formed by the vertex $\chi(\es)$ together with the vertices $\chi(\{u\})$ for $v \not\sim u$, and the clique $\DDD$ formed by the vertices $\chi(\{v,w\})$ for $v \sim w$.  Moreover, $\chi(\es)$ is adjacent to no vertex in $\DDD$.  Because of the present assumption on $p=\chi(\{v\})$, there exists a bipartition $C$, $D$ of $V$ as in Condition~[*] (with $v_0=v$), and so (b) does not hold.

If $p=\chi(\{v,w\})$, for some link $\{v,w\}$ in $E$, then $\chi(\{v,w\})$ is adjacent to $\chi(\{v\})$, to $\chi(\{w\})$, to all $\chi(\{v,u\})$ for $v \sim u$, to all $\chi(\{w,u\})$ for $w \sim u$ and to all $\chi(\{u\})$ for $v \not\sim u$ and $w \not\sim u$, and to no more vertex.  Because $\chi(\{v,w\})$ must be adjacent to $|V|$ vertices, the neighborhoods $N(v)$ of $v$ and $N(w)$ of $w$ in $\Gamma$ must be disjoint and the nodes linked to neither $v$ nor $w$ must form a stable set; also, a node in $\{v\}\cup N(v)$ is never linked to a node linked to neither $v$ nor $w$.  Two cliques appear: the clique $\CCC$ formed by the vertex $\chi(\{v\})$, the vertices $\chi(\{v,u\})$ for $u \in N(v)$ and the vertices $\chi(\{u\})$ for $u$ linked to neither $v$ nor $w$, and the clique $\DDD$ formed by the vertices $\chi(\{w\})$ and $\chi(\{w,u\})$ for $u \in N(w)$.  Moreover, $\chi(\{v\})$ is adjacent to no vertex in $\DDD$.  Because of the present assumption on $p=\chi(\{v,w\})$, there exists a bipartition $C$, $D$ of $V$ as in Condition~[*] (with $v_0=v$).

\smallskip

(c)~$\implies$~(d) The implication is trivial.

\smallskip

(d)~$\implies$~(a).  Any combinatorial automorphism of the polytope $P(\Gamma)$ is also an automorphism of the skeleton of $P(\Gamma)$.
\qed\end{proof}

To any automorphism $\alpha$ of the graph $\Gamma$, we now associate a combinatorial automorphism $F(\alpha)$ of the polytope $P(\Gamma)$ which fixes the vertex $\chi(\es)$.  Because $\alpha$ permutes the nodes $u$ of $\Gamma$, it induces a permutation $\alpha'$ of the characteristic vectors $\chi(\{u\})$.  As the latter vectors form the canonical  basis of $\R^V$, there is a unique linear permutation $\alpha''$ of $\R^V$ which extends $\alpha'$, and of course $\alpha''$ fixes $\chi(\es)$, the origin.  Because the canonical basis is orthonormal, $\alpha''$ is an isometry.  Moreover, for any link $\{v,w\}$ of the graph $\Gamma$, we see that $\alpha''$ maps the vertex $\chi(\{v,w\})$ of $P(\Gamma)$ onto the vertex $\chi(\{\alpha(v),\alpha(w)\})$ (because $\chi(\{v,w\})=\chi(\{v\}) + \chi(\{w\})$ and $\alpha''$ is linear).  Thus the linear permutation $\alpha''$ stabilizes the set of vertices of the polytope $P(\Gamma)$, and so $\alpha''$ induces a combinatorial automorphism of the polytope which fixes $\chi(\es)$; we denote this combinatorial automorphism by $F(\alpha)$.  

\begin{proposition}\label{propF}\sl
The mapping $F$ defined just above is an injective homomorphism from the group $A$ of automorphisms of the graph $\Gamma$ into the group $B$ of combinatorial automorphisms of the polytope $P(\Gamma)$.
When the graph $\Gamma$ satisfies Condition~[*], the mapping $F$ is a group isomorphism from $A$ to $B$, and moreover any combinatorial automorphism of $P(\Gamma)$ is induced by some isometry of $P(\Gamma)$.
\end{proposition}

\begin{proof}
By its construction, the mapping $F: A \to B$ is an injective homomorphism of groups.  To show that $F$ is surjective when $\Gamma$ satisfies Condition~[*], consider any combinatorial automorphism $\beta$ of the polytope $P(\Gamma)$.  By Lemma~\ref{lemma}, $\beta$ fixes $\chi(\es)$.  Hence $\beta$ stabilizes the neighborhood of $\chi(\es)$, and induces an automorphism $\beta^*$ of the graph induced on the neighborhood.  As the latter graph is isomorphic to the complement graph $\bar\Gamma$ (Corollary~\ref{cor1}(ii)), $\beta^*=\alpha'$ for some automorphism $\alpha$ of the graph $\Gamma$ (the notation $\alpha'$ was introduced just before the statement we are now proving).  We show $\beta=F(\alpha)$.  First, $\beta$ and $F(\alpha)$ are combinatorial automorphisms fixing $\chi(\es)$ which have the same action on the vertices $\chi(\{u\})$ for $u \in V$.  Second, for any combinatorial automorphism $\gamma$ fixing $\chi(\es)$, the action of $\gamma$ on the vertices of the form $\chi(e)$, for $e\in E$, is determined by the action of $\gamma$ on the vertices $\chi(\{u\})$, for $u\in V$.  Indeed, if $e=\{v,w\}$, then $\chi(e)$ is the only remaining vertex of the 2-dimensional face containing the vertices $\chi(\es)$, 
$\chi(\{v\})$ and $\chi(\{w\})$ (the latter face is defined by the affine inequality $\sum_{u\in V\setminus\{v,w\}} x_u \ge 0$).
Finally, by its definition, $F(\alpha)$ is always the restriction to the set of vertices of an isometry of $P(\Gamma)$.
\qed\end{proof}

The before last argument in the proof shows in fact that any permutation of the set of vertices of $P(\Gamma)$ which fixes the vertex $\chi(\es)$ and maps the vertices of any 2-dimensional face to the vertices of some 2-dimensional face is a combinatorial automorphism.  It would be nice to characterize the graphs $\Gamma$ for which any automorphism of the skeleton of the polytope $P(\Gamma)$ necessarily is a combinatorial automorphism of $P(\Gamma)$.  Proposition~\ref{prop_auto_exc} below provides partial answers.

\begin{example}\label{ex_6cycle}\rm
Let the graph $\Gamma$ be a six-cycle.  Then there is a non-identical automorphism $\rho$ of the skeleton of $P(\Gamma)$ which fixes $\chi(\es)$ and is not a combinatorial automorphim of $P(\Gamma)$.  To obtain such an automorphism, let $\rho$ fix all the vertices $\chi(\{u\})$ for $u \in V$, and map the vertex $\chi(e)$ for any link $e$ in $E$ to the vertex $\chi(f)$, with $f$ the link opposite to $e$ in the six-cycle.  
\end{example}

No other graph $\Gamma$ which is a cycle leads to a special automorphism of the skeleton as in Example~\ref{ex_6cycle}: the assertion follows from the next proposition. 
The \textsl{neighborhood} of a link $\{u,w\}$ of the graph $\Gamma$ is the set
$$
N(\{u,w\}) \;=\; \{ v\in V \setminus \{u,w\} \st v \sim u \text{ or } v \sim w\}.
$$
By Proposition~\ref{prop1}(iii), for $v\in V$ and $e\in E$, we have $v\in N(e)$ exactly when the vertices $\chi(v)$ and $\chi(e)$ of $P(\Gamma)$ are not adjacent.

\begin{proposition}\label{prop_auto_exc}\sl
If there exists some automorphism of the skeleton of $P(\Gamma)$ which fixes $\chi(\es)$ and is not a combinatorial automorphim of $P(\Gamma)$, then the graph $\Gamma$ contains two links $e$ and $f$ satisfying
\begin{enumerate}[\qquad\rm(i)]\itemsep0.3em 
\item $e$ and $f$ are disjoint, and there is no link from any node of $e$ to any node of $f$;
\item $N(e) = N(f)$.
\end{enumerate} 
When the graph $\Gamma$ does not contain any four-cycle, the converse also holds.
\end{proposition}

\begin{proof}
Let $\gamma$ be some automorphism of the skeleton of $P(\Gamma)$ which fixes $\chi(\es)$ and is not a combinatorial automorphim of $P(\Gamma)$.  Thus $\gamma$ stabilizes the set $\NNN(\chi(\es))$ of neighbours of $\chi(\es)$, and $\gamma$ induces an automorphism of the graph induced by the skeleton on $\NNN(\chi(\es))$.  Because the induced graph is isomorphic to the complement of $\Gamma$ (Corollary~\ref{cor1}), there exists some automorphism $\alpha$ of $\Gamma$ such that the combinatorial automorphism $F(\alpha)$ and $\gamma$ acts the same way on $\NNN(\chi(\es))$; they moreover fix $\chi(\es)$ (by our assumption on $\gamma$ and the definition of $F(\alpha)$). Setting $\beta$ equal to the product $\big(F(\alpha)\big)^{-1} \circ \gamma$, we get an automorphism of the skeleton which fixes $\chi(\es)$ and each vertex in $\NNN(\chi(\es))$.  
As $\gamma$ is not a combinatorial automorphism, $\beta$ differs from the identity.  Hence $\beta$ must send some vertex $\chi(e)$ to some vertex $\chi(f)$, where $e$ and $f$ are distinct links in $\Gamma$. 
Consequently, $\chi(e)$ and $\chi(f)$ are adjacent to the same vertices in $\NNN(\chi(\es))$.  By Proposition~\ref{prop1}(iii), $e$ and $f$ satisfy (i) and (ii) in the statement.

Conversely, when $\Gamma$ does not contain any four-cycle, the graph induced by the skeleton on $\{\chi(d) \st d\in E\}$ is complete.  Then the permutation which exchanges $\chi(e)$ and $\chi(f)$ (for $e$ and $f$ as in the statement), while leaving all other vertices of $P(\Gamma)$ fixed, is an automorphism of the skeleton; it is not a combinatorial automorphism (because it fixes $\chi(\es)$ and all $\chi(\{u\})$ for $u \in V$).
\qed\end{proof}

If some graph $\Gamma$ contains two links $e$ and $f$ as in Proposition~\ref{prop_auto_exc}, it is easy to build an \textsl{augmented} graph $\Gamma'$ having the same automorphism group as $\Gamma$ but containing no such pair of links: it suffices to add for any node $v$ of $\Gamma$ two new nodes $v'$ and $v''$ with two new links $\{v,v'\}$ and $\{v',v''\}$ (in case two nodes of $\Gamma$ have the same neighborhood only if they are unlinked, it even suffices to add for each node $v$ only one new node $v'$ and one new link $\{v,v'\}$).  

\section{Building a Binary Polytope for any Group}
\label{sect_Polytope_from_Group}

Given any finite group $G$, we now build a convex polytope $P_G$ as in Theorem~\ref{theorem}---however, we first leave Property~(v) out of our discussion.  As recalled in Section~\ref{sect_Graph_from_Group}, there exists a graph $\Gamma=\Gamma(G)$ whose automorphism group is $G$.  
If $\Gamma$ does not satisfy Condition~[*] of Section~3,  
we replace the graph $\Gamma$ with its complement but keep the notation $\Gamma$ for the resulting graph.  Notice that now the graph $\Gamma$ satisfies Condition~[*], except for the graph having only one node, and for three particular pairs of complementary graphs on $2$, $3$ or $4$ nodes respectively---each time, a path and its complement.  In the first case, we take for our convex polytope $P_G$ the polytope having only one vertex (at the origin); in each of the three other particular cases, the automorphism group of the graph(s) is the cyclic group of order $2$ and we take for our polytope $P_G$ the segment in $\R^1$  with endpoints $0$ and $1$.  When the graph $\Gamma$ satisfies Condition~[*], we may apply Corollaries~\ref{cor1}, \ref{prop2} and Proposition~\ref{propF}: the polytope $P_G=P(\Gamma(G))$ satisfies Properties~(i)--(iv) in Theorem~\ref{theorem}.

Let us discuss how large the resulting polytope $P_G=P(\Gamma(G))$ is with respect to the order $n$ of the group $G$.  First, recall that for any graph $\Gamma$ the dimension of $P(\Gamma)$ equals the number of nodes in $\Gamma$ (Corollary~\ref{prop2}).  We saw in Section~\ref{sect_Graph_from_Group} that there is a graph $\Gamma(G)$ with at most $2\,n$ nodes, except for the three exceptional groups $C_3$, $C_4$ and $C_5$.  
So apart from three exceptional groups, a careful choice of the graph $\Gamma$ delivers a polytope $P_G=P(\Gamma(G)$ of dimension at most $2\,n$.  
For $C_3$, there is a graph with $9$ nodes, for $C_4$ with 10 nodes, for $C_5$ with 15 nodes such that all three graphs satisfy Condition~[*].  It might be that another construction delivers polytopes with the desired properties and dimension less that 9, 10 and 15 respectively.
Furthermore, in view of the result of \cite{Babai_Godsil_1982} recalled in Section~\ref{sect_Graph_from_Group}, we can even build the  polytope $P_G$ in a space of dimension at most $n$ when $G$ does not  belong to the list of exceptional groups given in Section~\ref{sect_Graph_from_Group}.

What about the number of vertices of $P_G=P(\Gamma(G))$?   Assuming that the dimension of the polytope has been minimized, we might want to decrease the number of vertices as much as possible.  
In other words, having minimized the number of nodes in the graph $\Gamma$ with automorphism group $G$, we next want to minimize the number of links.  Characterizations of such minimal graphs exist for certain groups, see for instance \cite{McCarthy_Quintas_1979}.
If we agree to increase the dimension of the polytope $P_G$ to (at most) $2\,n+1$, where again $n=|G|$, we may use a result of  \cite{Babai_1981}: if $g$ is the minimum number of generators of $G$ and $G$ is neither $C_3$, $C_4$ or $C_5$, there is a graph $\Gamma$ with automorphism group $G$ and at most $2\,n+1$ nodes and $2n(g+1)$ links (because $\Gamma$ contains three-cycles, it satisfies Condition~[*]).  Our construction then produces a polytope of dimension at most $2\,n+1$ and with at most $2\,n(g+2)+2$ vertices.

We now establish Theorem~\ref{theorem}, this time taking into account Property~(v) on the automorphism group of the skeleton.  Assume the given group $G$ is of order at least $3$. To prove the full theorem, we again start by taking a graph $\Gamma$ whose automorphism group is $G$.  Next, as we did at the end of Section~\ref{sect_Polytope_from_Graph}, we build the augmented graph $\Gamma'$.  As $\Gamma'$ always satisfies Condition~[*] and has no pair of links $e$, $f$ as in Proposition~\ref{prop_auto_exc}, the polytope $P(\Gamma')$ satisfies all properties in Theorem~\ref{theorem} (as seen from  Corollaries~\ref{cor1}, \ref{prop2}, Lemma~\ref{lemma} and Propositions~\ref{propF}, \ref{prop_auto_exc}).

As announced in the introduction, we now deduce that for any finite group $G$ there is a directed graph whose asymmetric traveling-salesman polytope has its automorphism group isomorphic to $G$.  This follows at once from Theorem~\ref{theorem} combined with a nice result of \cite{Billera_Sarangarajan_1996}, which states that any binary, convex polytope is affinely isomorphic to the asymmetric traveling-salesman polytope of some directed graph.  

\begin{acknowledgements}
We would like to thank Selim Rexhep, a past Ph.D.\ student, for his many interesting questions and in particular the one asking for the existence of a convex polytope for any finite group.  We thank also two anonymous reviewers for their useful comments on a preliminary version.
\end{acknowledgements} 


\end{document}